\documentclass{amsart}

\usepackage{amsmath}
\usepackage{amstext}
\usepackage{amssymb}

\usepackage{amsthm}

\theoremstyle{plain}
\newtheorem{theorem}{Theorem}[section]
\newtheorem{lemma}[theorem]{Lemma}
\newtheorem{corollary}[theorem]{Corollary}

\newtheorem{refthm}{Theorem}
\newtheorem*{ryabykhthm}{Ryabykh's Theorem}
\newtheorem*{cauchygreenthm}{Cauchy-Green Theorem}

\theoremstyle{definition}
\newtheorem{definition}[theorem]{Definition}

\numberwithin{equation}{section}

\newcommand{\conj}[1]{\overline{#1}}
\DeclareMathOperator{\sgn}{sgn}
\DeclareMathOperator{\pv}{p.v.}
\DeclareMathOperator{\Rp}{Re}

\begin{document}


\title[Extremal Problems and an Extension of Ryabykh's Theorem]{Extremal Problems in Bergman Spaces and an Extension of 
Ryabykh's Theorem}
\author{Timothy Ferguson}
\address{Department of Mathematics\\1326 Stevenson Center\\Vanderbilt University\\Nashville, TN 37240}
\email{timothy.j.ferguson@vanderbilt.edu}
\thanks{Thanks to Peter Duren for his help in editing the manuscript.}

\begin{abstract}
We study linear extremal problems in the Bergman space $A^p$ of the unit disc 
for $p$ an even integer.  
Given a functional on the dual space of $A^p$ with 
representing kernel $k \in A^q$, where $1/p + 1/q = 1$, 
we show that if the Taylor coefficients of 
$k$ are sufficiently small, then the extremal function $F \in H^{\infty}$.  
We also show that if $q \le q_1 < \infty$, then 
$F \in H^{(p-1)q_1}$ if and only if $k \in H^{q_1}$. 
These results extend and provide a partial converse to a theorem of Ryabykh.
\end{abstract}

\maketitle

An analytic function $f$ in the unit disc ${\mathbb{D}}$ is said to belong to the 
Bergman space $A^p$ if 
\begin{equation*}  \|f\|_{A^p} = \left\{ \int_{{\mathbb{D}}} |f(z)|^p d\sigma(z)\right\}^{1/p} < \infty. \end{equation*}
Here $\sigma$ denotes normalized area measure, so that $\sigma({\mathbb{D}})=1.$ 
For $1<p<\infty$, each functional $\phi \in (A^p)^*$ has a unique 
representation 
\begin{equation*}
\phi(f) = \int_{{\mathbb{D}}} f \conj{k}\, d\sigma,
\end{equation*}
for some $k \in A^q$, where $q = p/(p-1)$ is the conjugate index.  
The function $k$ is called the kernel of the functional $\phi.$

In this paper we study the extremal problem of maximizing 
$\Rp\phi(f)$ among all functions $f \in A^p$ of unit norm.   
If $1<p<\infty$, then an extremal function always exists and is unique. 
However, to find it explicitly is in general a difficult problem, and 
few explicit solutions are known. Here we consider the problem of 
determining whether the kernel being ``well-behaved'' implies that 
the extremal function is also ``well-behaved.'' A known result 
in this direction is Ryabykh's theorem, which states that if the kernel is 
actually in the Hardy space $H^q$, then the extremal function must be in 
the Hardy space $H^p$.  In \cite{me}, we gave a proof of Ryabykh's 
theorem based on general properties of extremal functions in
uniformly convex spaces.

In this paper, we obtain a sharper version of Ryabykh's theorem
 in the case where $p$ is an even integer.  Our results are:
\begin{itemize}
\item For $q \le q_1 < \infty$, the extremal function 
$F \in H^{(p-1)q_1}$ if and only if the kernel $k \in H^{q_1}$.
\item If the Taylor coefficients of $k$ are ``small enough,'' then 
$F \in H^\infty$.
\item The map sending a kernel $k \in H^q$ to 
its extremal function $F\in A^p$ is a continuous map from $H^q \setminus 0$ 
into $H^p$. 
\end{itemize}

Our proofs rely heavily on Littlewood-Paley theory, and seem to require that 
$p$ be an even integer.  It is an open problem whether the results hold 
without this assumption.

\section{Extremal Problems and Ryabykh's Theorem}\label{intro}
We begin with some notation.  
If $f$ is an 
analytic function, $S_n f$ denotes its $n^{th}$ Taylor polynomial 
at the origin. Lebesgue area measure is denoted by $dA$, and 
$d\sigma$ denotes normalized area measure. 

If $h$ is a measurable function in the unit disc, the principal 
value of its integral is
\begin{equation*}
\pv \int_{{\mathbb{D}}} h\, dA = 
\lim_{r\rightarrow 1} \int_{r{\mathbb{D}}} h\, dA,
\end{equation*}
if the limit exists.  

We now recall some basic facts about Hardy and Bergman spaces.  For proofs 
and further information, see \cite{D_Hp} and \cite{D_Ap}. 
Suppose that $f$ is analytic in the unit disc.  For $0 < p < \infty$ and 
$0 < r < 1,$ the integral mean of $f$ is 
$$M_p(f,r)=
\bigg\{ \frac{1}{2\pi} \int_{0}^{2\pi} |f(re^{i\theta})|^p d\theta 
\bigg\}^{1/p}.$$
If $p=\infty$, we write
$$
M_\infty(f,r) = \max_{0\le \theta < 2\pi} |f(re^{i\theta})|.$$
For fixed $f$ and $p$, the integral means are increasing 
functions of $r.$  If $M_p(f,r)$ is bounded we say that $f$ is in 
the Hardy space $H^p.$  
For any function $f$ in $H^p,$ the radial limit 
$f(e^{i\theta}) = \lim_{r \rightarrow 1^-} f(re^{i\theta})$ 
exists for almost every $\theta.$   
 An $H^p$ function is uniquely determined by 
the values of its boundary function on any set of positive measure.  
The space $H^p$ is a Banach space with norm 
$$\| f \|_{H^p} = \sup_r M_p(f,r) = \|f(e^{i\theta})\|_{L^p}.$$
It is useful to regard $H^p$ as a subspace of 
$L^p({\mathbb{T}}),$ where ${\mathbb{T}}$ denotes the unit circle. 
For $0<p<\infty$, if $f \in H^p$, then $f(re^{i\theta})$ 
converges to $f(e^{i\theta})$ in $L^p$ norm as $r \rightarrow 1$.

For $1<p<\infty$, 
the dual space $(H^p)^*$ is isomorphic to $H^q$, where $1/p + 1/q = 1$,
with an element
$k\in H^q$ representing the functional $\phi$ 
defined by 
\[\phi(f) = \frac{1}{2\pi}\int_{0}^{2\pi} f(e^{i\theta})
\conj{k(e^{i\theta})}\,d\theta.\]
 This isomorphism is not an isometry unless $p=2$, 
but it is true that $\|\phi\|_{(H^p)^*} \le \|k\|_{H^q} \le 
C \|\phi\|_{(H^p)^*}$ for some constant $C$ depending only on $p$. 
If $f\in H^p$ for $1<p<\infty$, then $S_n f \rightarrow f$ in 
$H^p$ as $n \rightarrow \infty$. 
The Szeg\H{o} projection $S$ maps each function $f \in L^1({\mathbb{T}})$ into 
a function analytic in ${\mathbb{D}}$ defined by 
\begin{equation*}
Sf(z) = \frac{1}{2\pi}\int_0^{2\pi} \frac{f(e^{it})}{1-e^{-it}z} dt.
\end{equation*}
It leaves $H^1$ functions fixed and maps $L^p$ boundedly onto $H^p$ for 
$1 < p < \infty$.   
If $f \in L^p$ for $1<p<\infty$ and 
$\displaystyle f(\theta) = \sum_{n=-\infty}^\infty a_n e^{in\theta},$ 
then $\displaystyle Sf(z) = \sum_{n=0}^\infty a_n z^n.$

For $1 < p < \infty,$ the dual of the Bergman space $A^p$ is isomorphic to $A^q$, 
where $1/p + 1/q = 1,$ and $k\in A^q$ represents the functional 
defined by $\phi(f) = \int_{{\mathbb{D}}} f(z)\conj{k(z)}\,d\sigma(z)$. 
Note that this isomorphism is actually conjugate-linear.  
It is not an isometry unless $p=2$, 
but  
if the functional $\phi \in (A^p)^*$ is represented by the 
function $k \in A^q$, then
\begin{equation}\label{A_q_isomorphism}
\| \phi \|_{(A^p)^*} \le \| k \|_{A^q} \le C_p \| \phi \|_{(A^p)^*}
\end{equation}
where $C_p$ is a constant depending only on $p$. 
We remark that $H^p \subset A^p$, and in fact 
$\|f\|_{A^p} \le \|f\|_{H^p}.$ 
If $f \in A^p$ for $1<p<\infty$, then $S_n f \rightarrow f$ in $A^p$ 
as $n\rightarrow \infty$.

In this paper the only Bergman spaces we consider are those with
$1<p<\infty$. 
For a given linear 
functional $\phi \in (A^p)^*$ such that $\phi \ne 0$,
we investigate the extremal problem of finding a function $F \in A^p$ 
with norm 
$\|F\|_{A^p} = 1$ for which
\begin{equation}\label{norm1}
\Rp \phi(F) = \sup_{\|g\|_{A^p}=1} \Rp \phi(g) = \| \phi \|.
\end{equation}
Such a function $F$ is called an extremal function, and 
we say that $F$ is an extremal function for a function $k \in A^q$ 
if $F$ solves problem 
\eqref{norm1} for the functional $\phi$ with kernel $k$. 
This problem has been studied  by
Vukoti\'{c} \cite{Dragan}, Khavinson and Stessin \cite{Khavinson_Stessin},
and Ferguson \cite{me}, among others.  
Note that for $p=2$, the extremal function is $F = k/\|k\|_{A^2}.$ 

A closely related problem is that of finding $f\in A^p$
such that 
$\phi(f) = 1$ and 
\begin{equation}\label{value1}
\|f\|_{A^p} = \inf_{\phi(g) = 1} \|g\|_{A^p}.
\end{equation}
If $F$ solves the problem \eqref{norm1}, then $\frac{F}{\phi(F)}$ 
solves the problem 
\eqref{value1}, and if $f$ solves \eqref{value1}, then 
$\frac{f}{\|f\|}$ solves \eqref{norm1}.
When discussing either of these problems, we always assume that $\phi$ is 
not the zero functional; in other words, that $k$ is not identically $0$. 

The problems \eqref{norm1} and \eqref{value1} each have a unique solution 
when $1<p<\infty$ (see \cite{me}, Theorem 1.4). Also, for every 
function $f \in A^p$ such that $f$ is not identically $0$, there 
is a unique $k \in A^q$ such that 
$f$ solves problem \eqref{value1} for $k$ 
(see \cite{me}, Theorem 3.3).  This implies that for each $F \in A^p$ with 
$\|F\|_{A^p} = 1$, there is some nonzero $k$ such that $F$ solves problem 
\eqref{norm1} for $k.$  Furthermore, any two such kernels $k$ are positive 
multiples of each other. 

The Cauchy-Green theorem is an important tool in this paper. 
\begin{cauchygreenthm}
If $\Omega$ is a region in the plane with 
piecewise smooth boundary and $f\in C^1({\overline{\Omega}})$, then
$$ \frac{1}{2 i} \int_{\partial \Omega} f(z) \, dz = \int_{\Omega} 
\frac{\partial}{\partial \conj{z}} f(z) \,dA(z),$$
where $\partial \Omega$ denotes the 
boundary of $\Omega$.
\end{cauchygreenthm}

The next result is an important characterization of extremal functions in 
$A^p$ for $1<p<\infty$ (see \cite{Shapiro_Approx}, p.55). 
\begin{refthm}\label{integral_extremal_condition} 
Let $1 < p < \infty$ and let $\phi \in (A^p)^*$.   
A function $F \in A^p$ with $\|F\|_{A^p} = 1$ satisfies 
$$\Rp \phi(F) = \sup_{\|g\|_{A^p} =1} \Rp \phi(g) = \| \phi \|$$
if and only if
$$\int_{{\mathbb{D}}} h |F|^{p-1} \conj{\sgn F}  \, d\sigma = 0$$ for all 
$h \in A^p$ with $\phi(h) = 0.$  
If $F$ satisfies the above conditions, then 
$$\int_{{\mathbb{D}}} h |F|^{p-1} \conj{\sgn F}\,  d\sigma 
= \frac{\phi(h)}{\| \phi \|}$$
for all $h\in A^p.$
\end{refthm}

Ryabykh's theorem relates extremal problems 
in Bergman spaces to Hardy spaces.  It says that if the kernel for a 
linear functional is not only in $A^q$ but also in $H^q$, then the extremal 
function is not only in $A^p$ but in $H^p$ as well.  

\begin{ryabykhthm}\label{ryabykh_thm}
Let $1 < p < \infty$ and let $1/p + 1/q = 1.$  Suppose that 
$\phi \in (A^p)^*$ and $\phi(f) = \int_{{\mathbb{D}}} f \conj{k} \, d\sigma$ 
for some $k \in H^q$.
Then the solution $F$ to the extremal problem \eqref{norm1} belongs to 
$H^p$ and satisfies 
\begin{equation}\label{ryabkh_estimate}
\|F\|_{H^p} \le \Bigg\{ \bigg[ \max(p-1,1)
\bigg]\frac{C_p\|k\|_{H^q}}{\| k \|_{A^q}}\Bigg\}^{1/(p-1)},
\end{equation}
where $C_p$ is the constant in \eqref{A_q_isomorphism}.
\end{ryabykhthm}
Ryabykh\cite{Ryabykh} proved that $F \in H^p.$  The bound \eqref{ryabkh_estimate} was 
proved in \cite{me}, by a variant of Ryabykh's proof.

As a corollary Ryabykh's theorem implies that the solution to the 
problem \eqref{value1} is in $H^p$ as well.
Note that the constant $C_p \rightarrow \infty$ as 
$p \rightarrow 1$ or $p \rightarrow \infty$.

To obtain our results, including a generalization of Ryabykh's theorem, we will need 
the following technical lemmas.  Their 
proofs, which involve Littlewood-Paley theory, are deferred to the 
end of the paper. 

\begin{lemma}\label{kleq} Let $p$ be an even integer.  Let $f \in H^p$ and let 
$h$ be a polynomial. Then 
\begin{equation*}
\pv \int_{{\mathbb{D}}} |f|^{p-1} \conj{\sgn f}f'h\, d\sigma = 
\lim_{n\rightarrow \infty} \int_{{\mathbb{D}}} |f|^{p-1} \conj{\sgn f}(S_n f)'h \,d\sigma.
\end{equation*}
  
\end{lemma}

\begin{lemma}\label{klb}
Suppose that $1 < p_1 < \infty$ and $1<p_2,p_3 \le \infty$, 
and also that   
$$1 = \frac{1}{p_1} + \frac{1}{p_2} + \frac{1}{p_3}.$$
Let $f_1 \in H^{p_1}$, $f_2 \in H^{p_2}$, and $f_3 \in H^{p_3}.$  Then 
\begin{equation*}
\left|\pv \int_{{\mathbb{D}}} \conj{f_1} f_2 f_3' \, d\sigma \right| \le C 
\|f_1\|_{H^{p_1}} \|f_2\|_{H^{p_2}} \|f_3\|_{H^{p_3}} 
\end{equation*}
where $C$ depends only on $p_1$ and $p_2$. 
(Implicit is the claim that the principal 
value exists.)  Moreover, if $p_3 < \infty$, then
\begin{equation*}
\pv \int_{{\mathbb{D}}} \conj{f_1} f_2 f_3'\, d\sigma = 
\lim_{n\rightarrow\infty} \int_{{\mathbb{D}}} \conj{f_1} f_2 (S_n f_3)'\, d\sigma.
\end{equation*}

\end{lemma}

\section{The Norm-Equality}
Let $p$ be an even integer and let $q$ be its conjugate exponent.
Let $k \in H^q$ and let $F$ be the extremal function for $k$ over $A^p$.  
We will denote by $\phi$ the functional associated with $k$. Let 
$F_n$ be the extremal function for $k$ when the extremal problem is posed 
over $P_n$, the space of polynomials of degree at most $n$.  Also, let
\begin{equation}\label{K_eq}
K(z) = \frac{1}{z} \int_0^z k(\zeta)\, d\zeta, 
\end{equation}
so that $(zK)' = k.$  
During proof of Ryabykh's 
theorem in \cite{me}, an important step is to show that 
\begin{equation*}
\frac{1}{2\pi} \int_0^{2\pi} |F_n(e^{i\theta})|^p d\theta = 
\frac{1}{2\pi \|\phi_{|P_n}\|} \int_{0}^{2 \pi}
F_n \left[\left(\frac{p}{2}\right)\conj{k} +\left(1-\frac{p}{2}\right) \conj{K}\right]\,d\theta,
\end{equation*}
(see \cite{me}, p. 2652). 
We will now derive a similar result for $F$:
\begin{theorem}\label{norm_formula_ext}
Let $p$ be an even integer, let $k \in H^q$, 
and let $F\in A^p$ be the extremal function for $k.$  Then 
\begin{equation*}
\frac{1}{2\pi} \int_0^{2\pi} |F(e^{i\theta})|^p h(e^{i\theta}) d\theta = 
\frac{1}{2\pi \|\phi\|} \int_{0}^{2 \pi}
F \left[\left(\frac{p}{2}\right)h\conj{k} +\left(1-\frac{p}{2}\right) (zh)'\conj{K}\right]\,d\theta,
\end{equation*}
for every polynomial $h$.
\end{theorem}

\begin{proof}
Since Ryabykh's theorem says that $F\in H^p$, we have 
\begin{equation*}
\frac{1}{2\pi} \int_0^{2\pi} |F(e^{i\theta})|^p h(e^{i\theta}) d\theta = 
\lim_{r\rightarrow 1} \frac{i}{2\pi} \int_{\partial (r{\mathbb{D}})} |F(z)|^p h(z)z\, d\conj{z},
\end{equation*}
where $h$ is any polynomial. 
Apply the Cauchy-Green theorem to transform the right-hand side into
\begin{equation*}
\pv \frac{1}{\pi} \int_{{\mathbb{D}}} \left((zh)'F + \frac{p}{2}zhF'\right)|F|^{p-1}\conj{\sgn F}\,dA(z).
\end{equation*}
Invoking Lemma \ref{kleq} with $zh$ in place of $h$ shows that this limit equals 
\begin{equation*}
\lim_{n\rightarrow\infty} \frac{1}{\pi} 
\int_{{\mathbb{D}}} \left((zh)'F + \frac{p}{2}zh(S_n F)'\right)|F|^{p-1}\conj{\sgn F}\,dA(z).
\end{equation*}
Since $(zh)'F + \frac{p}{2}zh(S_n F)'$ is in $A^p$, 
we may apply 
Theorem \ref{integral_extremal_condition} to reduce the last expression 
to  
\begin{equation}\label{eq_normeq1}
\lim_{n\rightarrow\infty} \frac{1}{\pi \|\phi\|} 
\int_{{\mathbb{D}}} \left((zh)'F + \frac{p}{2}zh(S_n F)'\right) \conj{k} \,dA(z). 
\end{equation}
Recall that we have defined
 $K(z) = \tfrac{1}{z}\int_0^z k(\zeta)\, d\zeta.$
To prepare for a reverse application of the Cauchy-Green theorem, we rewrite 
the integral in \eqref{eq_normeq1} as
\begin{equation*}
\begin{split}
\frac{1}{\pi \|\phi\|} 
\int_{{\mathbb{D}}} \bigg[
\frac{\partial}{\partial \conj{z}}\left\{(zh)'F\conj{zK}\right\} &+ 
\frac{p}{2} \frac{\partial}{\partial z}\left\{zhS_n(F) \conj{k}\right\} \\
&- 
\frac{p}{2} \frac{\partial}{\partial \conj{z}}\left\{(zh)'S_n(F) \conj{zK}\right\} \bigg] dA(z).
\\
\end{split}
\end{equation*}
Now this equals
\begin{equation*}
\begin{split}
\lim_{r\rightarrow 1} \frac{1}{\pi\|\phi\|} \int_{r{\mathbb{D}}} \bigg[
\frac{\partial}{\partial \conj{z}}\left\{(zh)'F\conj{zK}\right\} &+ 
\frac{p}{2} \frac{\partial}{\partial z}\left\{zhS_n(F) \conj{k}\right\} \\  
&- \frac{p}{2} \frac{\partial}{\partial \conj{z}}\left\{(zh)'S_n(F) \conj{zK}\right\} \bigg] dA(z).
\\
\end{split}
\end{equation*}
We apply the Cauchy-Green theorem to show that this equals
\begin{equation*}\begin{split}
\lim_{r \rightarrow 1} \Biggl[ 
\frac{1}{2\pi i \|\phi\|} \int_{\partial (r{\mathbb{D}})} (zh)'F \conj{zK}\, dz &+ 
\frac{ip}{4\pi  \|\phi\|} \int_{\partial (r{\mathbb{D}})} zh S_n(F) \conj{k}\, d\conj{z} 
\\ &- 
\frac{p}{4\pi i \|\phi\|} \int_{\partial (r{\mathbb{D}})} (zh)'S_n(F) \conj{zK}\, dz \Biggr]. \\
\end{split}\end{equation*}
Since $F$ is in $H^p$ and both $k$ and $K$ are in $H^q$, the above limit 
equals
\begin{equation*}
\begin{split} 
&\frac{1}{2\pi i \|\phi\|} \int_{\partial {\mathbb{D}}} (zh)'F \conj{zK}\, dz + 
\frac{ip}{4\pi  \|\phi\|} \int_{\partial {\mathbb{D}}} zh S_n(F) \conj{k}\, d\conj{z}\\
&\qquad \qquad - 
\frac{p}{4\pi i \|\phi\|} \int_{\partial {\mathbb{D}}} (zh)'S_n(F) \conj{zK}\, dz 
\\
&= \frac{1}{2\pi\|\phi\|} 
\int_0^{2\pi} (zh)'F \conj{K} + S_n(F)\left(\frac{p}{2}h\conj{k} - \frac{p}{2}(zh)'\conj{K}\right) d\theta. 
\\
\end{split}
\end{equation*}
We let $n \rightarrow \infty$ in the above expression to reach the 
desired conclusion. 
\end{proof}

Taking $h=1$, we have the following corollary, which we call the 
``norm-equality''.
\begin{corollary}\label{norm_equality}{\rm\bf (The Norm-Equality).}
Let $p$ be an even integer, let $k \in H^q$, and let $F$ be the extremal 
function for $k.$  Then 
\begin{equation*}
\frac{1}{2\pi} \int_0^{2\pi} |F(e^{i\theta})|^p d\theta = 
\frac{1}{2\pi \|\phi\|} \int_{0}^{2 \pi}
F \left[\left(\frac{p}{2}\right)\conj{k} +\left(1-\frac{p}{2}\right) \conj{K}\right]\,d\theta.
\end{equation*}
\end{corollary}

The norm-equality is useful mainly because it yields the following 
theorem. 
\begin{theorem}\label{thm_cont}
Let $p$ be an even integer. 
Let $\{k_n\}$ be a sequence of $H^q$ functions, and let 
$k_n \rightarrow k$ in $H^q$.
Let $F_n$ be the $A^p$ extremal function for $k_n$ and let 
$F$ be the $A^p$ extremal function for $k.$  Then $F_n \rightarrow F$ in $H^p.$
\end{theorem}
Note that Ryabykh's theorem shows that each $F_n \in H^p$, 
and that $F \in H^p$.  
But because the operator taking a kernel to its extremal function is not 
linear, one cannot apply the closed graph theorem to conclude
that $F_n \rightarrow F$.  

To prove Theorem \ref{thm_cont} we will use the following lemma involving the 
notion of uniform convexity. A Banach 
space $X$ is called {\it uniformly convex} if 
for each $\epsilon > 0$, there is a $\delta > 0$ such that for all 
$x,y \in X$ with $\|x\| = \|y\| = 1,$ 
$$\left\| \tfrac12(x+y) \right\| > 1 - \delta \qquad \text{ implies } 
\qquad 
\|x-y\| < \epsilon.
$$
An equivalent definition is that if 
$\{x_n\}$ and $\{y_n\}$ are sequences in $X$ such that 
$\|x_n\| = \|y_n\| = 1$ for all $n$ and $\|x_n + y_n\| \rightarrow 2$ 
then $\|x_n - y_n\| \rightarrow 0$. 
This concept was introduced by Clarkson in \cite{Clarkson}.  See also 
\cite{me},  where it is applied to extremal problems. 
To apply the lemma,
we use the fact that the space $H^p$ is uniformly convex for $1<p<\infty.$ 
By $x_n \rightharpoonup x$, we mean that
$x_n$ approaches $x$ weakly.
\begin{lemma}\label{uc_convergence_lemma}
Suppose that $X$ is a uniformly convex Banach space, that $x \in X,$ and 
that $\{x_n\}$ is a sequence of elements of $X$.  If $x_n \rightharpoonup x$ 
and $\|x_n\| \rightarrow \|x\|$, then $x_n \rightarrow x$ in $X$.
\end{lemma}
This lemma is known.  For example, it is contained in Exercise 15.17 in 
\cite{Hewitt_Stromberg}. 

\begin{proof}[Proof of Theorem]
We will first show  that 
$F_n \rightharpoonup F$ in $H^p$ (that is, $F_n$ converges to $F$ weakly in 
$H^p$).  Next we will use this fact and the norm-equality to show that 
$\|F_n\|_{H^p} \rightarrow \|F\|_{H^p}$. By the lemma, 
it will then follow that $F_n \rightarrow F$ in $H^p$. 

To prove that $F_n \rightharpoonup F$ in $H^p$, note that    
Ryabykh's theorem says that 
$\|F_n\|_{H^p} \le C (\|k_n\|_{H^q}/\|k_n\|_{A^q})^{1/(p-1)}.$  
Let $\alpha = \inf_n \|k_n\|_{A^q}$  
and $\beta =\sup_n \|k_n\|_{H^q}$.  
Here $\alpha > 0$ because by assumption none of the $k_n$ are identically 
zero, and they approach $k$, which is not identically $0$. 
Therefore 
$\|F_n\|_{H^p} \le C(\beta/\alpha)^{1/(p-1)}$, and the sequence 
$\{F_n\}$ is bounded in $H^p$ norm. 

Now, suppose that $F_n \not\rightharpoonup F.$ Then 
there is some $\psi \in (H^p)^*$ such that $\psi(F_n) \not \rightharpoonup\psi(F).$ 
This implies  
$|\psi(F_{n_j}) - \psi(F)| \ge \epsilon$ for some $\epsilon >0$ and some 
subsequence $\{F_{n_j}\}$. 
But since the sequence $\{F_n\}$ is bounded in $H^p$ norm, 
the Banach-Alaoglu theorem implies that some subsequence 
of $\{F_{n_j}\}$, which we will also denote by $\{F_{n_j}\}$, 
converges weakly in $H^p$ to some function 
$\widetilde{F}.$  Then $|\psi(\widetilde{F}) - \psi(F)| \ge \epsilon$.
Now $k_n \rightarrow k$ in $A^q$, and it is proved in \cite{me} that 
this implies $F_n \rightarrow F$ in $A^p$, which implies  
$F_n(z) \rightarrow F(z)$ for all $z \in {\mathbb{D}}.$ 
Since point evaluation is a bounded linear 
functional on $H^p,$ we have that  
$F_{n_j}(z) \rightarrow \widetilde{F}(z)$ for all $z\in {\mathbb{D}}$,
which means that $\widetilde{F}(z) = F(z)$ for all $z \in {\mathbb{D}}.$ 
 But this contradicts the assumption
that $\psi(\widetilde{F}) \ne \psi(F).$  Hence $F_n \rightharpoonup F$. 

Let $\phi_n$ be the functional with kernel $k_n$, and let $\phi$ be the 
functional with kernel $k$. 
To show that $\|F_n\|_{H^p} \rightarrow \|F\|_{H^p},$ recall that the 
norm-equality says
\begin{equation*}
\frac{1}{2\pi} \int_0^{2\pi} |F_n(e^{i\theta})|^p d\theta = 
\frac{1}{2\pi \|\phi_n\|} \int_{0}^{2 \pi}
F_n \left[\left(\frac{p}{2}\right)\conj{k_n} +\left(1-\frac{p}{2}\right) \conj{K_n}\right]\,d\theta.
\end{equation*}
But, if $h$ is any function analytic in ${\mathbb{D}}$ and $H(z)=(1/z)
\int_0^z h(\zeta)d\zeta$, it can be shown that 
$\|H\|_{H^q} \le \|h\|_{H^q}$ (see \cite{me}, proof of Theorem 4.2). 
Since $k_n \rightarrow k$ in 
$H^q$, it follows that $K_n \rightarrow K$ in $H^q$.
Also, $k_n \rightarrow k$ in $A^p$ implies that 
$\|\phi_n\| \rightarrow \|\phi\|$. In addition, 
$\|F_n\|_{H^p} \le C$ for some constant 
$C,$ and $F_n \rightharpoonup F,$ so the 
right-hand side of the above equation approaches 
\begin{equation*}
\frac{1}{2\pi \|\phi\|} \int_{0}^{2 \pi}
F \left[\left(\frac{p}{2}\right)\conj{k} +\left(1-\frac{p}{2}\right) \conj{K}\right]\,d\theta
= \frac{1}{2\pi} \int_0^{2\pi} |F(e^{i\theta})|^p d\theta.
\end{equation*}
In other words, $\|F_n\|_{H^p} \rightarrow \|F\|_{H^p},$ and so by 
Lemma \ref{uc_convergence_lemma} we conclude that 
$F_n \rightarrow F$ in $H^p$.  
\end{proof}

\section{Fourier Coefficients of $|F|^p$}

Theorem \ref{norm_formula_ext} can also be used to gain information about 
the Fourier coefficients of $|F|^p$, where $F$ is the extremal function.
In particular, it leads to a criterion for $F$ to be in $L^\infty$ in terms 
of the Taylor coefficients of the kernel $k$.

\begin{theorem}\label{fourier_fp}
Let $p$ be an even integer.  Let $k \in H^q$, let 
$F$ be the $A^p$ extremal function for $k$, and define 
$K$ by equation \eqref{K_eq}. Then for any integer $m \ge 0$, 
\begin{equation*}
\frac{1}{2\pi}\int_0^{2\pi} |F(e^{i\theta})|^p e^{im\theta} d\theta = 
\frac{1}{2\pi \|\phi\|} \int_0^{2\pi} Fe^{im\theta} 
\left[ \left(\frac{p}{2}\right) \conj{k} + \left( 1-\frac{p}{2}\right)(m+1)\conj{K}\right]
d\theta.
\end{equation*}
\end{theorem}
\begin{proof}
Take $h(e^{i\theta}) = e^{im\theta}$ in Theorem 
\ref{norm_formula_ext}. 
\end{proof}

This last formula can be applied to obtain estimates on the size of the 
Fourier coefficients of $|F|^p.$

\begin{theorem}\label{Fp_bound}
Let $p$ be an even integer. Let $k \in A^q$, and let $F$ be the $A^p$ extremal 
function for $k$.  Let 
\begin{equation*}
b_m = \frac{1}{2\pi}\int_0^{2\pi}|F(e^{i\theta})|^p e^{-im\theta} d\theta,
\end{equation*}
and let 
\begin{equation*}
k(z) = \sum_{n=0}^\infty c_n z^n.
\end{equation*}
Then, for each $m \ge 0,$
\begin{equation*}
|b_m| = |b_{-m}| \le \frac{p}{2\|\phi\|}  \|F\|_{H^2}
\left[ \sum_{n=m}^\infty |c_n|^2 \right]^{1/2}.
\end{equation*}
\end{theorem}

\begin{proof}
The theorem is trivially true if $k \not \in H^2$, so we may assume 
that $k \in A^2 \subset A^q.$
Let $F(z) = \sum_{n=0}^\infty a_n z^n$. 
Since $F \in H^p,$ and $p\ge 2$, we have $F \in H^2.$  Now, using Theorem
\ref{fourier_fp}, we find that
\begin{equation*}
\begin{split}
b_{-m} &= 
\frac{1}{2\pi} \int_0^{2\pi}|F(e^{i\theta})|^p e^{im\theta} d\theta \\ &= 
\frac{1}{2\pi\|\phi\|} \int_0^{2\pi} (Fe^{im\theta}) \left[ \left(\frac{p}{2}\right) 
\conj{k} + 
\left( 1 - \frac{p}{2}\right)(m+1)\conj{K} \right] d\theta \\
&=  \frac{1}{2\pi\|\phi\|} \int_0^{2\pi} 
\left[ \sum_{n=0}^\infty a_{n}e^{i(n+m)\theta} \right]
 \Bigg[ \sum_{j=0}^\infty \left( \left(\frac{p}{2}\right) \conj{c_j} + 
\frac{m+1}{j+1} \left( 1 - \frac{p}{2} \right) \conj{c_j} \right) e^{-ij\theta} \Bigg] d\theta
\\
&=   \frac{1}{\|\phi\|} 
\left| \sum_{n=0}^\infty a_{n} 
\left( \left(\frac{p}{2}\right) \conj{c_{n+m}} + 
\frac{m+1}{n+m+1} \left( 1 - \frac{p}{2} \right) \conj{c_{n+m}} \right)\right|.
\\
\end{split}
\end{equation*}
The Cauchy-Schwarz inequality now gives 
\begin{equation*}
\begin{split}
|b_{-m}| &\le  \frac{1}{\|\phi\|} 
\left[ \sum_{n=0}^\infty |a_n|^2 \right]^{1/2}
\left[ \sum_{n=m}^\infty \left| \left(\frac{p}{2}\right) \conj{c_n} + 
\frac{m+1}{n+1} \left( 1 - \frac{p}{2} \right) \conj{c_n} \right|^2
\right]^{1/2}\\
&\le \frac{p}{2\|\phi\|}  \left[ \sum_{n=0}^\infty |a_{n}|^2 \right]^{1/2}
\left[ \sum_{n=m}^\infty |c_n|^2 \right]^{1/2}.
\end{split}
\end{equation*}
Since 
\begin{equation*} 
 \left[ \sum_{n=0}^\infty |a_{n}|^2 \right]^{1/2} = \|F\|_{H^2}
\end{equation*}
the theorem follows.
\end{proof}

The estimate in Theorem \ref{Fp_bound} can be used to obtain information 
about the size of $|F|^p$ and $F$, as in the following corollary. 

\begin{corollary}\label{order_k_linfty_condition}
If $c_n = O(n^{-\alpha})$ for some $\alpha > 3/2$, then $F \in H^\infty$.
\end{corollary}
\begin{proof}
First observe that
\begin{equation*}
\sum_{n=m}^\infty (n^{-\alpha})^2 \le \int_{m-1}^\infty x^{-2\alpha} dx 
= \frac{(m-1)^{1-2\alpha}}{2\alpha-1}.
\end{equation*}
By hypothesis it follows that
\begin{equation*}
\left[ \sum_{n=m}^\infty |c_n|^2 \right]^{1/2} = O(m^{(1-2\alpha)/2}).
\end{equation*}
Thus, 
Theorem \ref{Fp_bound} shows that 
$b_m = O(m^{(1-2\alpha)/2})$. Therefore 
$\{b_m\} \in \ell^1$ if $\alpha > 3/2$.  But 
$\{b_m\} \in \ell^1$ implies $|F|^p \in L^\infty$, which implies 
$F \in H^\infty$.
\end{proof}

In fact, $\{b_m\}\in \ell^1$ implies that $|F|^p$ is continuous in 
$\overline{{\mathbb{D}}}$, but this does not necessarily mean $F$ will be continuous 
in $\overline{{\mathbb{D}}}$.
There is a result similar to Corollary \ref{order_k_linfty_condition} 
in  \cite{Khavinson_Stessin}, where 
the authors show that if the kernel $k$ is a polynomial, 
or even a rational function with 
no poles in $\overline{{\mathbb{D}}}$, then $F$ is H\"{o}lder continuous in 
$\overline{{\mathbb{D}}}$. Their technique relies on deep regularity 
results for partial differential equations. Our result only shows that 
$F \in H^\infty$, but it applies to a broader class of kernels.

\section{Relations Between the Size of the Kernel and Extremal Function}

In this section we show that if $p$ is an even integer and 
$q\le q_1 < \infty$, then the 
extremal function $F \in H^{(p-1)q_1}$ if and only if the kernel 
$k \in H^{q_1}.$ For $q_1=q$ the statement reduces to Ryabykh's theorem 
and its previously unknown converse. 
The following theorem  is crucial to the proof.
\begin{theorem}\label{extremal_bound}
Let $p$ be an even integer and let $q = p/(p-1)$ be its conjugate exponent. 
Let $F \in A^p$ be the extremal function corresponding to the kernel 
$k \in A^q$.
Suppose that 
$k \in H^{q_1}$ for some $q_1$ with $q \le q_1 < \infty$, 
and that 
$F \in H^{p_1}$, for some $p_1$ with $p \le p_1 < \infty$.   
Define $p_2$ by 
$$
\frac{1}{q_1} + \frac{1}{p_1} + \frac{1}{p_2} = 1.$$
 If $p_2 < \infty$, then for every trigonometric polynomial $h$ we have  
\begin{equation*}
\left| \int_0^{2\pi} |F|^p h(e^{i\theta})\, d\theta \right| \le 
C\frac{\|k\|_{H^{q_1}}}{\|k\|_{A^q}} \|F\|_{H^{p_1}} 
\|h\|_{L^{p_2}},
\end{equation*}
where $C$ is some constant depending only on $p$, $p_1$, and $q_1$.
\end{theorem}
The excluded case $p_2 = \infty$ occurs if and only if 
$q=q_1$ and $p=p_1.$  The theorem is then a trivial consequence of Ryabykh's 
theorem.
\begin{proof}[Proof of Theorem]
First let $h$ be an analytic polynomial.
In the proof of Theorem \ref{norm_formula_ext}, we showed that 
\begin{equation}\label{41eq1}
\begin{split}
 \frac{1}{2\pi} \int_0^{2\pi} |F(e^{i\theta})|^p h(e^{i\theta})\, d\theta = 
\lim_{n\rightarrow\infty} \frac{1}{\pi\|\phi\|} 
\int_{{\mathbb{D}}} \left((hz)'F + \frac{p}{2}hz(S_n F)'\right) \conj{k} \,dA(z).\\
\end{split}
\end{equation}
An application of Lemma \ref{klb} gives
\begin{equation*}
\lim_{n\rightarrow\infty} \int_{{\mathbb{D}}} hz(S_n F)' \conj{k}\, dA = 
\pv \int_{{\mathbb{D}}} hzF' \conj{k}\, dA,
\end{equation*}
so that the right-hand side of equation \eqref{41eq1} becomes 
\begin{equation*}
\frac{1}{\pi\|\phi\|} 
\pv \int_{{\mathbb{D}}} \left((hz)'F + \frac{p}{2}hzF'\right)\conj{k}\,dA(z).
\end{equation*}
Apply Lemma \ref{klb} separately to the two parts of the integral to 
conclude that its absolute value is bounded by
\begin{equation*}
C \frac{1}{\|\phi\|}\|k\|_{H^{q_1}} \|f\|_{H^{p_1}} \|h\|_{H^{p_2}},
\end{equation*}
where $C$ is a constant depending only on $p_1$ and $q_1$. 
Since 
$$\frac{1}{\|\phi\|} \le \frac{C_p}{\|k\|_{A^q}}$$ 
by equation \eqref{A_q_isomorphism},
this gives the desired result for the special case where $h$ is an analytic polynomial.

Now let $h$ be an arbitrary trigonometric polynomial. 
Then $h = h_1 + \conj{h_2},$ where $h_1$ and 
$h_2$ are analytic polynomials, and $h_2(0)=0.$    
Note that the Szeg\H{o} projection $S$ is bounded from 
$L^{p_2}$ into $H^{p_2}$ because $1< p_2< \infty.$  Thus, 
$$ \|h_1\|_{H^{p_2}} = \|S(h)\|_{H^{p_2}} \le C \|h\|_{L^{p_2}}.$$
Also, 
$$ \|h_2\|_{H^{p_2}} = \|z S(e^{-i\theta} \conj{h})\|_{H^{p_2}} = 
\|S(e^{-i\theta}\conj{h})\|_{H^{p_2}} \le C \|e^{-i\theta} \conj{h}\|_{L^{p_2}}
= C\|h\|_{L^{p_2}},$$
and so
$$ \|h_1\|_{H^{p_2}} + \|h_2\|_{H^{p_2}} \le 
C\|h\|_{L^{p_2}}.$$ 
Therefore, by what we have already shown,
\begin{equation*}
\begin{split}
\left|\int_{0}^{2\pi} |f(e^{i\theta})|^p h(e^{i\theta}) d\theta \right| &= 
\left|\int_{0}^{2\pi} |f(e^{i\theta})|^p(h_1(e^{i\theta})+\conj{h_2(e^{i\theta})}) d\theta 
\right| \\ 
&\le \left| \int_{0}^{2\pi} |f|^p h_1 \, d\theta\right| + 
\left| \conj{\int_{0}^{2\pi} |f|^p h_2\, d\theta} \right| \\ 
&\le 
 C\frac{\|k\|_{H^{q_1}}}{\|k\|_{A^q}} \|f\|_{H^{p_1}}(\|h_1\|_{H^{p_2}} + \|h_2\|_{H^{p_2}}) 
\\&\le
C\frac{\|k\|_{H^{q_1}}}{\|k\|_{A^q}} \|f\|_{H^{p_1}} \|h\|_{L^{p_2}}.\qedhere
\end{split}
\end{equation*}
\end{proof}

For a given $q_1$, we will apply the theorem just proved with $p_1$ chosen 
as $p_1 = pp_2'$, where $p_2'$ is the conjugate exponent to $p_2.$  This 
will allow us to bound the $H^{p_1}$ norm of $f$ solely in terms of 
$\|\phi\|$ and $\|k\|_{H^{q_1}}$. 

\begin{theorem}\label{r_ext}
 Let $p$ be an even integer, and let $q$ be its conjugate 
exponent. Let $F\in A^p$ be the extremal function for a kernel 
$k \in A^q$. 
If, for $q_1$ such that $q \le q_1 < \infty$, the kernel  
$k \in H^{q_1},$ then $F \in H^{p_1}$ for $p_1=(p-1)q_1.$  In fact, 
\begin{equation*}
\|F\|_{H^{p_1}} \le C\left(
\frac{\|k\|_{H^{q_1}}}{\|k\|_{A^q}}\right)^{1/(p-1)},
\end{equation*}
where $C$ depends only on $p$ and $q_1$.
\end{theorem}
\begin{proof}
The case $q_1 = q$ is Ryabykh's theorem, so we assume $q_1 > q.$ 
Set $p_1 = (p-1)q_1.$ Then $p_1 > p = (p-1)q.$  
Choose $p_2$ so that $$
\frac{1}{q_1} + \frac{1}{p_1} + \frac{1}{p_2} = 1.$$  
This implies that $p_2 = p_1/(p_1-p)$, and so its conjugate exponent
$p_2' = p_1/p$.  Note that $1<p_2<\infty$. 
Let $F_n$ denote the extremal function corresponding to the kernel 
$S_n k$, which does not vanish identically if $n$ is chosen sufficiently large. 
Since $S_n k$ is a polynomial, 
$F_n$ is in $H^\infty$  (and thus $F_n \in H^{p_1}$) by Corollary
\ref{order_k_linfty_condition}. Hence
for any trigonometric polynomial $h$, Theorem \ref{extremal_bound} yields  
\begin{equation*}
\left|\frac{1}{2\pi}\int_0^{2\pi} |F_n|^p h(e^{i\theta}) d\theta \right|\le 
C \frac{\|S_n k\|_{H^{q_1}}}{\|S_n k\|_{A^q}} 
\|F_n\|_{H^{p_1}} \|h\|_{L^{p_2}}. 
\end{equation*}
Since the trigonometric polynomials are dense in $L^{p_2}(\partial {\mathbb{D}})$, 
taking the supremum over all trigonometric polynomials $h$ 
with $\|h\|_{L^{p_2}} \le 1$ gives 
\begin{equation*}
 \| |F_n|^p\|_{L^{p_2'}} \le
C \frac{\|S_n k\|_{H^{q_1}}}{\|S_n k\|_{A^q}} \|F_n\|_{H^{p_1}},
\end{equation*}
which implies
\begin{equation*}
\begin{split}
\|F_n\|_{H^{p_1}}^p &= 
 \left\{ \frac{1}{2\pi} \int_0^{2\pi} 
(|F_n(e^{i\theta})|^p)^{p_2'} d\theta \right\}^{1/p_2'} =
\| |F_n|^p\|_{L^{p_2'}} \\ &\le
C \frac{\|S_n k\|_{H^{q_1}}}{\|S_n k\|_{A^q}} \|F_n\|_{H^{p_1}},
\end{split}
\end{equation*}
since $pp_2' = p_1$.
Because $\|F_n\|_{H^{p_1}} < \infty$, 
we may divide both sides of the inequality by $\|F_n\|_{H^{p_1}}$ to obtain
\begin{equation*}
\|F_n\|_{H^{p_1}}^{p-1} \le C  \frac{\|S_n k\|_{H^{q_1}}}{\|S_n k\|_{A^q}},
\end{equation*}
where $C$ depends only  on $p$ and $q_1$. In other words,
\begin{equation*}
\left( \frac{1}{2\pi} \int_0^{2\pi} |F_n(re^{i\theta})|^{p_1} d\theta \right)^{(p-1)/p_1} 
\le C\frac{\|S_n k\|_{H^{q_1}}}{\|S_n k\|_{A^q}}
\end{equation*} 
for all $r<1$ and for all $n$ sufficiently large. 
Note that $S_n k \rightarrow k$ in 
$H^{q_1}$ and in $A^q$.
Since $S_n k \rightarrow k$ in $A^q,$ Theorem 3.1 in \cite{me} says that 
$F_n \rightarrow F$ in $A^p$, and thus $F_n \rightarrow F$ 
uniformly on compact subsets of ${\mathbb{D}}$.  
Thus, letting $n \rightarrow \infty$ in the last inequality gives 
\begin{equation*}
\left(  \frac{1}{2\pi} \int_0^{2\pi} |F(re^{i\theta})|^{p_1} d\theta \right)^{(p-1)/p_1} 
\le C\frac{\|k\|_{H^{q_1}}}{\|k\|_{A^q}}
\end{equation*}
for all $r<1.$  In other words,  
\begin{equation*}
\|F\|_{H^{p_1}} \le \left(C\frac{\|k\|_{H^{q_1}}}{\|k\|_{A^q}}\right)^{1/(p-1)}. \qedhere
\end{equation*}
\end{proof} 

Recall from Section \ref{intro} that 
a function $F \in A^p$ with unit norm has a 
corresponding kernel $k \in A^q$ such that $F$ is the extremal function for 
$k$, and this kernel is uniquely determined up to a positive multiple.
Theorem \ref{r_ext}  says that if $p$ is an even integer and a 
kernel $k$ belongs not only to the 
Bergman space $A^q$ but also to the Hardy space $H^{q_1}$ for some 
$q_1$ where $q \le q_1 < \infty$, then the $A^p$ extremal function $F$
associated with it   
is actually in $H^{p_1}$ for $p_1 = (p-1)q_1 \ge p.$ 
It is natural to ask 
whether the converse is true.  In other words, if $F\in H^{p_1}$ for some 
$p_1$ with $p \le p_1 < \infty$, 
must it follow that the corresponding kernel belongs to 
$H^{q_1}$?    
The following theorem says that this is indeed the case. 
\begin{theorem} Suppose $p$ is an even integer and let $q$ be its conjugate 
exponent. 
Let $F \in A^p$ with $\|F\|_{A^p}=1$, 
and let $k$ be a kernel such that $F$ is the extremal function for $k$. 
If $F \in H^{p_1}$ for some $p_1$ with $p \le p_1 < \infty$, then   
$k \in H^{q_1}$ for $q_1 = p_1/(p-1)$, and
\begin{equation*}
\frac{\|k\|_{H^{q_1}}}{\|k\|_{A^q}} \le
C \|F\|_{H^{p_1}}^{p-1},
\end{equation*} 
where $C$ is a constant depending only on $p$ and $p_1$. 
\end{theorem}

\begin{proof}
Let $h$ be a polynomial and let $\phi$ be the functional in $(A^p)^*$ 
corresponding to $k$. Then by Theorem \ref{integral_extremal_condition}, 
\begin{equation*}
\begin{split}
\frac{1}{\|\phi\|} \int_{{\mathbb{D}}} \conj{k(z)} (zh(z))' d\sigma &= 
\int_{{\mathbb{D}}} |F(z)|^{p-1}\sgn(\conj{F(z)}) (zh(z))' d\sigma \\
&= \int_{{\mathbb{D}}} \conj{F^{p/2}} F^{(p/2)-1} (zh(z))' d\sigma .
\end{split}
\end{equation*}
By hypothesis, $F^{p/2} \in H^{(2p_1)/p}$ and 
$F^{(p/2)-1} \in H^{2p_1/(p-2)}.$ A simple calculation shows that 
\begin{equation*}
\frac{1}{q_1'} = \frac{q_1-1}{q_1} = \frac{p_1-p+1}{p_1}
\end{equation*}
and thus \begin{equation*}
\frac{p}{2p_1} + \frac{p-2}{2p_1} + \frac{1}{q_1'}=1.
\end{equation*}
Now we will apply the first part of Lemma \ref{klb} with 
$f_1 = F^{p/2}$ and $f_2 = F^{(p/2)-1}$ and $f_3=zh$, and with 
$2p_1/p$ in place of $p_1$, and $2p_1/(p-2)$ in place of 
$p_2$, and $q_1'$ in place of $p_3.$  Note that this is permitted since 
$1 < 2p_1/p < \infty$, and $1 < q_1' < \infty$, 
and $1 < 2p_1/(p-2) \le \infty$. 
(In fact, we even know that $2p_1/(p-2) < \infty$ unless $p=2$, 
which is a trivial case since then $F=k/\|k\|_{A^2}$.) With these choices, 
Lemma \ref{klb} 
gives

\begin{equation*}
\begin{split} \left|
 \int_{{\mathbb{D}}} \conj{F^{p/2}} F^{(p/2)-1} (zh(z))' d\sigma \right|
& \le C\|F^{p/2}\|_{H^{2p_1/p}} 
\|F^{p/2-1}\|_{H^{2p_1/(p-2)}} \|zh\|_{H^{q_1'}}\\
&=  C\|F\|_{H^{p_1}}^{p/2} \|F\|_{H^{p_1}}^{(p-2)/2} \|h\|_{H^{q_1'}}\\
&= C\|F\|_{H^{p_1}}^{p-1} \|h\|_{H_{q_1'}}.
\end{split}
\end{equation*}

Since 
$$ \left|\int_{{\mathbb{D}}} \conj{k(z)} (zh(z))' d\sigma \right|
 \le 
C\|\phi\|\|F\|_{H^{p_1}}^{p-1} \|h\|_{H_{q_1'}}$$ for all polynomials $h$, we 
may define a continuous linear functional $\psi$ on $H^{q_1'}$ such that 
$$
\psi(h) = \int_{{\mathbb{D}}} \conj{k(z)} (zh(z))' d\sigma
$$
for all analytic polynomials $h$.  
Then $\psi$ has an associated kernel in $H^{q_1}$, which we 
will call $\widetilde{k}.$ 
Thus, for all $h \in H^{q_1'}$, we have 
$$ \psi(h) = \frac{1}{2\pi}\int_0^{2\pi} \conj{\widetilde{k}(e^{i\theta})} 
h(e^{i\theta}) d\theta. $$
But then the Cauchy-Green theorem gives  
\begin{equation}\label{kernel_trick_eqn}
\begin{split}
 \int_{{\mathbb{D}}} \conj{k(z)} (zh(z))' \,d\sigma &= 
\psi(h) \\&= \frac{1}{2\pi}\int_{\partial {\mathbb{D}}} \conj{\widetilde{k}(e^{i\theta})} h(e^{i\theta}) \,d\theta=  
\frac{i}{2\pi}\int_{\partial {\mathbb{D}}} \conj{\widetilde{k}(z)} h(z)z \,d\conj{z}\\&= 
\lim_{r\rightarrow 1} \frac{i}{2\pi}
\int_{\partial (r{\mathbb{D}})} \conj{\widetilde{k}(z)} h(z)z \,d\conj{z} =
\lim_{r\rightarrow 1} 
\int_{r{\mathbb{D}}} \conj{\widetilde{k}(z)} (zh(z))' \,d\sigma \\&= 
\int_{{\mathbb{D}}} \conj{\widetilde{k}(z)} (zh(z))' \,d\sigma,
\end{split}
\end{equation}
where $h$ is any analytic polynomial. 

Now, 
for any polynomial $h(z),$ define the polynomial $H(z)$ so that 
\begin{equation*} H(z) = \frac{1}{z} \int_{0}^z h(\zeta) \, d\zeta.
\end{equation*}
Then substituting $H(z)$ for $h(z)$ in equation \eqref{kernel_trick_eqn}, 
and using the fact that 
$(zH)' = h$, we have 
$$\int_{{\mathbb{D}}} \conj{\widetilde{k}(z)} h(z) d\sigma = 
 \int_{{\mathbb{D}}} \conj{k(z)} h(z) d\sigma$$
for every polynomial $h$.
But since the polynomials are dense in $A^p$, and $k$ and $\widetilde{k}$ are 
both in 
$A^q,$ which is isomorphic to the dual space of $A^p$, we must have that 
$k = \widetilde{k}$, and thus  $k \in H^{q_1}.$ 

Now for any polynomial $h$,
\begin{equation*}
\frac{1}{2\pi}\int_0^{2\pi} \conj{k(e^{i\theta})} h(e^{i\theta}) d\theta \le
C\|\phi\|\|F\|_{H^{p_1}}^{p-1} \|h\|_{H^{q_1'}},
\end{equation*}
and so 
\begin{equation*}
\frac{1}{2\pi}\int_0^{2\pi} \conj{k(e^{i\theta})} h(e^{i\theta}) d\theta \le
C\|k\|_{A^q} \|F\|_{H^{p_1}}^{p-1} \|h\|_{H^{q_1'}}
\end{equation*}
by inequality \eqref{A_q_isomorphism}.
But if $h$ is any trigonometric polynomial, 
\begin{equation*}
\begin{split}
\frac{1}{2\pi}\int_0^{2\pi} \conj{k(e^{i\theta})} h(\theta) d\theta &=
\frac{1}{2\pi}\int_0^{2\pi} \conj{k(e^{i\theta})} \left[S(h)(e^{i\theta})\right] d\theta \\ &\le
C\|k\|_{A^q} \|F\|_{H^{p_1}}^{p-1} \|S(h)\|_{H^{q_1'}} \\&\le
C\|k\|_{A^q} \|F\|_{H^{p_1}}^{p-1} \|h\|_{L^{q_1'}},
\end{split}
\end{equation*}
where $S$ denotes the Szeg\H{o} projection.
Taking the supremum over all trigonometric polynomials $h$ with 
$\|h\|_{L^{q_1'}} \le 1$ and 
dividing both sides of the inequality by $\|k\|_{A^q}$ we arrive at the 
required bound. 
\end{proof}

The main results of this section can be summarized in the following theorem.
\begin{theorem}
Suppose that $p$ is an even integer with conjugate exponent $q$.   
Let $k\in A^q$ and let $F$ be the $A^p$ extremal function associated 
with $k$. 
Let $p_1, q_1$ be a pair of 
numbers such that $q \le q_1 < \infty$ and 
$$p_1 = (p-1)q_1.$$
Then 
$F \in H^{p_1}$ if and only if $k \in H^{q_1}$.
More precisely,
\begin{equation*}
C_1\left(\frac{\|k\|_{H^{q_1}}}{\|k\|_{A^q}}\right)^{1/(p-1)} 
\le  \|F\|_{H^{p_1}} 
\le 
C_2\left(\frac{\|k\|_{H^{q_1}}}{\|k\|_{A^q}}\right)^{1/(p-1)} 
\end{equation*}
where $C_1$ and $C_2$ are constants that depend only on $p$ and $p_1$. 
\end{theorem}
Note that if $p_1 = (p-1)q_1$, then $q \le q_1 < \infty$ is 
equivalent to $p \le p_1 < \infty.$ 

\section{Proof of the Lemmas}
We now give the proofs of Lemmas \ref{kleq} and \ref{klb}.  These proofs 
are rather technical and require applications of maximal functions and 
Littlewood-Paley theory.

\begin{definition} 
For a function $f$ analytic in the unit disc, 
the Hardy-Littlewood maximal function is defined on the unit circle by  
\begin{equation*}
f^*(e^{i\theta}) = \sup_{0 \le r < 1} |f(re^{i\theta})|.
\end{equation*}
\end{definition}
The following is the simplest form of the Hardy-Littlewood maximal theorem 
(see for instance \cite{D_Hp}, p. 12).  
\begin{refthm}\label{hardy_littlewood} {\rm \bf (Hardy-Littlewood.)} 
If $f \in H^p$ for $0<p \le \infty$, then 
$f^* \in L^p$ and 
\begin{equation*}
\|f^*\|_{L^p} \le C \|f\|_{H^p},
\end{equation*}
where $C$ is a constant depending only on $p$.
\end{refthm}
Further results of a similar type may be found in \cite{Garnett}.

\begin{definition}
For a function $f$ analytic in the unit disc, the Littlewood-Paley function is
$$g(\theta,f) = \left\{ \int_0^1 (1-r)|f'(re^{i\theta})|^2 dr 
\right\}^{1/2}.
$$
\end{definition}
A key result of Littlewood-Paley theory is that the Littlewood-Paley function, 
like the Hardy-Littlewood maximal function, belongs to $L^p$ if and only if 
$f \in H^p.$  Formally, the result may be stated as follows 
(see \cite{Zygmund}, Volume 2, Chapter 14,
Theorems 3.5 and 3.19).  
\begin{refthm}\label{littlewood_paley} {\rm \bf (Littlewood-Paley.)}
For $1<p<\infty$,  there are constants 
$C_p$ and $B_p$ depending only on $p$ so that 
\begin{equation*}
\|g(\cdot,f)\|_{L^p} \le C_p \|f\|_{H^p} 
\end{equation*}
for all functions $f$ analytic in ${\mathbb{D}}$, and
\begin{equation*}
\|f\|_{H^p} \le B_p \|g(\cdot,f)\|_{L^p}
\end{equation*}
for all functions $f$ analytic in ${\mathbb{D}}$ such that $f(0)=0.$
\end{refthm}

We now apply the Littlewood-Paley theorem to obtain the following result, 
from which Lemmas \ref{kleq} and \ref{klb} will follow.   
\begin{theorem}\label{hardy_mult_int}
Suppose $1 < p_1, p_2 \le \infty$, and let $p$ be defined by 
$1/p = 1/p_1 + 1/p_2.$  
Suppose furthermore that $1 < p < \infty.$  If $f_1 \in H^{p_1}$ and 
$f_2 \in H^{p_2},$ and $h$ is defined by 
$$
h(z) = \int_0^z f_1(\zeta)f_2'(\zeta)\, d\zeta,$$
then $h \in H^p$ and $\|h\|_{H^p} \le C \|f_1\|_{H^{p_1}}   
\|f_2\|_{H^{p_2}}$, where $C$ depends only on $p_1$ and $p_2$.      
\end{theorem}
\begin{proof}
By the definitions of the Littlewood-Paley function and the 
Hardy-Littlewood maximal function, 
\begin{equation*}
\begin{split}
g(\theta, h) &= \left\{ \int_0^1 (1-r)|f_1(re^{i\theta})f_2'(re^{i\theta})|^2\, dr \right\}^{1/2}\\ &\le 
f_{1}^*(\theta)\left\{ \int_0^1 (1-r)|f_2'(re^{i\theta})|^2 \,dr \right\}^{1/2}\\
&= f_{1}^*(\theta) g(\theta,f_2).
\end{split}
\end{equation*}
Therefore, since $h(0)=0$, 
Theorem \ref{littlewood_paley} gives 
\begin{equation*}
\|h\|_{H^p} \le C \|g(\cdot, h)\|_{L^p}
\le C \|f_1^* \, g(\cdot, f_2)\|_{L^p}.
\end{equation*}
Applying first H\"{o}lder's inequality and then Theorem \ref{hardy_littlewood},
we infer that 
\begin{equation*}
\|h\|_{H^p} \le 
C\|f_{1}^*\|_{L^{p_1}} \|g(\cdot, f_2)\|_{L^{p_2}} \le 
C\|f_{1}\|_{H^{p_1}} \|g(\cdot, f_2)\|_{L^{p_2}}. \end{equation*}
If $p_2 < \infty$,  
Theorem \ref{littlewood_paley} allows us to conclude that 
\begin{equation*}
\|h\|_{H^p} \le 
C \|f_1\|_{H^{p_1}} \|f_2\|_{H^{p_2}}.
\end{equation*}
This proves the claim under the assumption that $p_2 < \infty$.

If $p_2 = \infty,$ then $p_1 < \infty$ by assumption. 
Integration by parts gives 
$$h(z) = f_1(z)f_2(z) - f_1(0)f_2(0) - \int_0^z f_2(\zeta)f_1'(\zeta)\,d\zeta.$$
The $H^p$ norm of the first term is bounded by 
$\|f_1\|_{H^{p_1}} \|f_2\|_{H^{p_2}}$, by H\"{o}lder's inequality.
The second term is bounded by $C\|f_1\|_{H^{p_1}} \|f_2\|_{H^{p_2}}$ 
for some $C$, since point evaluation is a bounded functional on 
Hardy spaces.  The $H^p$ norm of the 
last term is bounded by $C\|f_1\|_{H^{p_1}} \|f_2\|_{H^{p_2}}$,
by what we have already shown, and thus 
$\|h\|_{H^p} \le C \|f_1\|_{H^{p_1}} \|f_2\|_{H^{p_2}}$.  
\end{proof}

Theorem \ref{hardy_mult_int} will now be used together with the Cauchy-Green 
theorem to prove Lemmas \ref{klb} and \ref{kleq}. 
\begin{proof}[{\bf Proof of Lemma \ref{klb}}]
Define 
$$I_r = \int_{r{\mathbb{D}}} \conj{f_1} f_2 f_3' \,dA \ \quad\text{ and }\quad\ 
H(z) = \int_0^z f_2(\zeta) f_3'(\zeta) d\zeta.$$
Then Theorem \ref{hardy_mult_int} says that $H \in H^q$ and that  
$\|H\|_{H^q} \le C \|f_2\|_{H^{p_2}} \|f_3\|_{H^{p_3}}$,
where 
$\tfrac{1}{q} = \tfrac{1}{p_2} + \tfrac{1}{p_3}.$
By the Cauchy-Green formula, 
\begin{equation*}
I_r = \frac{i}{2}\int_{\partial(r{\mathbb{D}})} \conj{f_1(z)} H(z)\, d\conj{z}.
\end{equation*}
Since $1/p_1 + 1/q = 1$, 
H\"{o}lder's inequality gives
\begin{equation*}
|I_r| = \frac{1}{2} \left|\int_{\partial(r{\mathbb{D}})} \conj{f_1(z)} H(z)\,d\conj{z} \right| \le 
\pi M_{p_1}(f_1,r) M_q(H,r).
\end{equation*}
But since $\|H\|_{H^q} \le C \|f_2\|_{H^{p_2}} \|f_3\|_{H^{p_3}}$, 
this shows that 
\begin{equation*}
|I_r| 
\le C \|f_1\|_{H^{p_1}} \|f_2\|_{H^{p_2}} \|f_3\|_{H^{p_3}},
\end{equation*}
which bounds the principal value in question, assuming it exists.

To show that it exists, note that for $0<s<r$, the Cauchy-Green formula gives
\begin{equation*}
\begin{split}
2|I_r-I_s| &= \left|\int_{\partial(r{\mathbb{D}} - s{\mathbb{D}})} \conj{f_1(z)} H(z) \, d\conj{z} \right|\\ 
&= \left| \int_0^{2\pi} \left[r\conj{f_1(re^{i\theta})}H(re^{i\theta}) - 
s\conj{f_1(se^{i\theta})}H(se^{i\theta})\right] e^{-i\theta}\,d\theta \right| \\
&\le \left| \int_0^{2\pi} \conj{f_1(re^{i\theta})}\left(rH(re^{i\theta})-
sH(se^{i\theta})\right)e^{-i\theta}\,d\theta \right| \\ &\quad + 
\left| \int_{0}^{2\pi}
s\left(\conj{f_1(re^{i\theta})} - \conj{f_1(se^{i\theta})}\right)
H(se^{i\theta})\, e^{-i\theta}\,d\theta \right|.
\end{split}
\end{equation*}
We let $f_r(z) = f(rz).$ 
Then H\"{o}lder's inequality shows that the expression on the right
of the above inequality is at most 
\begin{equation*}
M_{p_1}(f_1,r)\|rH_r - sH_s\|_{H^q} + 
s \|(f_1)_r - (f_1)_s\|_{H^{p_1}} M_q(H,r).
\end{equation*}
Since $p_1 < \infty$ and  $q < \infty$, we know that
$(f_1)_r \rightarrow f_1$ in $H^{p_1}$ as $r \rightarrow 1,$ and 
$H_r \rightarrow H$ in $H^q$ as $r \rightarrow 1$
(see \cite{D_Hp}, p. 21). 
Thus  
the above quantity approaches $0$ as $r,s \rightarrow 1$, which shows that 
the principal value exists.

For the last part of the lemma, what was already shown gives 
\begin{equation*}
\begin{split}
\pv \int_{{\mathbb{D}}} \conj{f_1}f_2 f_3' d\sigma - \int_{{\mathbb{D}}} \conj{f_1}f_2(S_n f_3)' d\sigma &= 
\pv \int_{{\mathbb{D}}} \conj{f_1}f_2 (f_3-S_n f_3)' d\sigma \\ &\le
C \|f_1\|_{H^{p_1}} \|f_2\|_{H^{p_2}} \|f_3-S_n(f_3)\|_{H^{p_3}}.
\end{split} 
\end{equation*}
By assumption $p_3 > 1$. If also $p_3 < \infty$, then the 
right hand side approaches $0$ as 
$n\rightarrow \infty$, which finishes the proof. 
\end{proof}

\begin{proof}[{\bf Proof of Lemma \ref{kleq}}]
We know that $f^{p/2} \in H^2$ and $f^{(p/2)-1} \in H^{2p/(p-2)}.$  Since 
$h$ is a polynomial, we have $f^{(p/2)-1}h \in H^{2p/(p-2)}.$  
Also,
\begin{equation*}
\frac{1}{2} + \frac{p-2}{2p} + \frac{1}{p} = 1.
\end{equation*}
Thus, Lemma \ref{klb} with $f_1 = f^{p/2}$, and $f_2=f^{(p/2)-1}h$, 
and $f_3=f$ gives the result.  

\end{proof}

\providecommand{\bysame}{\leavevmode\hbox to3em{\hrulefill}\thinspace}
\providecommand{\MR}{\relax\ifhmode\unskip\space\fi MR }
\providecommand{\MRhref}[2]{%
  \href{http://www.ams.org/mathscinet-getitem?mr=#1}{#2}
}
\providecommand{\href}[2]{#2}

\end{document}